\documentclass[12pt,a4paper]{amsart}
\usepackage{amsmath,amstext,amssymb,amsthm,amsfonts}
\newcommand{\nc}{\newcommand}

\nc{\one}{\mbox{\bf 1}}
\nc{\invtensor}{\underset{\leftarrow}{\otimes}}
\nc{\const}{\operatorname{const}}

\nc{\ad}{\operatorname{ad}}
\nc{\tr}{\operatorname{tr}}
\nc{\tp}{\operatorname{top}}
\nc{\rank}{\operatorname{rank}}
\nc{\corank}{\operatorname{corank}}
\nc{\codim}{\operatorname{codim}}
\nc{\sdim}{\operatorname{sdim}}
\nc{\mult}{\operatorname{mult}}

\nc{\spn}{\operatorname{span}}
\nc{\Sym}{\operatorname{Sym}}
\nc{\sym}{\operatorname{sym}}
\nc{\id}{\operatorname{id}}
\nc{\Id}{\operatorname{Id}}
\nc{\Ree}{\operatorname{Re}}

\nc{\htt}{\operatorname{ht}}
\nc{\Ker}{\operatorname{Ker}}
\nc{\rker}{\operatorname{rKer}}
\nc{\im}{\operatorname{Im}}
\nc{\osp}{\mathfrak{osp}}
\nc{\sgn}{\operatorname{sgn}}
\nc{\F}{\operatorname{F}}
\nc{\Mod}{\operatorname{Mod}}
\nc{\Mat}{\operatorname{Mat}}
\nc{\Soc}{\operatorname{Soc}}
\nc{\Inj}{\operatorname{Inj}}
\nc{\Hom}{\operatorname{Hom}}
\nc{\End}{\operatorname{End}}
\nc{\supp}{\operatorname{supp}}
\nc{\Card}{\operatorname{Card}}
\nc{\Ann}{\operatorname{Ann}}
\nc{\Ind}{\operatorname{Ind}}
\nc{\Coind}{\operatorname{Coind}}
\nc{\wt}{\operatorname{wt}}
\nc{\ch}{\operatorname{ch}}
\nc{\Stab}{\operatorname{Stab}}
\nc{\Sch}{{\mathcal S}\mbox{\em ch}}
\nc{\Irr}{\operatorname{Irr}}
\nc{\Spec}{\operatorname{Spec}}
\nc{\Prim}{\operatorname{Prim}}
\nc{\Aut}{\operatorname{Aut}}
\nc{\Ext}{\operatorname{Ext}}

\nc{\Fract}{\operatorname{Fract}}
\nc{\gr}{\operatorname{gr}}
\nc{\deff}{\operatorname{def}}
\nc{\HC}{\operatorname{HC}}

\nc{\red}{\operatorname{red}}

\nc{\wdchi}{\widetilde{\chi}}
\nc{\wdH}{\widetilde{H}}
\nc{\wdN}{\widetilde{N}}
\nc{\wdM}{\widetilde{M}}
\nc{\wdO}{\widetilde{O}}
\nc{\wdR}{\widetilde{R}}
\nc{\wdS}{\widetilde{S}}
\nc{\wdV}{\widetilde{V}}

\nc{\wdC}{\widetilde{C}}

\nc{\Obj}{\operatorname{Obj}}
\nc{\Dglie}{\operatorname{{\mathcal D}glie}}
\nc{\Fin}{\operatorname{{\mathcal F}in}}

\nc{\Adm}{\operatorname{\mathcal{A}dm}}

\nc{\Sg}{{\cS(\fg)}}
\nc{\Shg}{{\cS(\fhg)}}
\nc{\Ug}{{\cU(\fg)}}
\nc{\Uhg}{{\cU(\fhg)}}
\nc{\Sh}{{\cS(\fh)}}
\nc{\Uh}{{\cU(\fh)}}
\nc{\Uhh}{{\cU(\fhh)}}
\nc{\Zg}{{{\mathcal{Z}}(\fg)}}

\nc{\Vir}{{\mathcal{V}ir}}

\nc{\tZg}{{\widetilde{\mathcal Z}({\mathfrak g})}}
\nc{\Zk}{{\mathcal Z}({\mathfrak k})}

\nc{\Up}{{\mathcal U}({\mathfrak p})}
\nc{\Ah}{{\mathcal A}({\mathfrak h})}
\nc{\Ag}{{\mathcal A}({\mathfrak g})}
\nc{\Ap}{{\mathcal A}({\mathfrak p})}
\nc{\Zp}{{\mathcal Z}({\mathfrak p})}
\nc{\cZ}{\mathcal Z}
\nc{\cS}{\mathcal S}
\nc{\cT}{\mathcal{T}}
\nc{\cY}{\mathcal Y}
\nc{\cA}{\mathcal A}
\nc{\cU}{\mathcal U}
\nc{\cH}{\mathcal H}
\nc{\cM}{\mathcal M}
\nc{\cL}{\mathcal L}
\nc{\cF}{\mathcal F}
\nc{\fg}{\mathfrak g}

\nc{\fo}{\mathfrak o}

\nc{\CO}{\mathcal O}
\nc{\Cl}{\mathcal {C}\ell}

\nc{\cR}{\mathcal{R}}
\nc{\bM}{\mathbf{M}}
\nc{\bL}{\mathbf{L}}
\nc{\bN}{\mathbf{N}}

\nc{\zq}{\mathpzc q}

\nc{\fl}{\mathfrak l}
\nc{\fn}{\mathfrak n}
\nc{\fm}{\mathfrak m}
\nc{\fp}{\mathfrak p}
\nc{\fh}{\mathfrak h}
\nc{\ft}{\mathfrak t}
\nc{\fk}{\mathfrak k}
\nc{\fb}{\mathfrak b}
\nc{\fs}{\mathfrak s}
\nc{\fB}{\mathfrak B}

\nc{\vareps}{\varepsilon}
\nc{\varesp}{\varepsilon}
\nc{\veps}{\varepsilon}

\nc{\fsl}{\mathfrak{sl}}
\nc{\fpsl}{\mathfrak{psl}}
\nc{\fgl}{\mathfrak{gl}}
\nc{\fso}{\mathfrak{so}}

\nc{\fpq}{\mathfrak{pq}}
\nc{\fq}{\mathfrak q}
\nc{\fsq}{\mathfrak{sq}}
\nc{\fpsq}{\mathfrak{psq}}

\nc{\fhg}{\hat{\fg}}
\nc{\fhn}{\hat{\fn}}
\nc{\fhh}{\hat{\fh}}
\nc{\fhb}{\hat{\fb}}
\nc{\hrho}{\hat{\rho}}

\nc{\hsl}{\hat{\fsl}}
\nc{\fpo}{\mathfrak{po}}
\nc{\dirlim}{\underset{\rightarrow}{\lim}\,}
\nc{\nen}{\newenvironment}
\nc{\ol}{\overline}
\nc{\ul}{\underline}
\nc{\ra}{\rightarrow}
\nc{\lra}{\longrightarrow}
\nc{\Lra}{\Longrightarrow}

\nc{\Lla}{\Longleftarrow}

\nc{\Llra}{\Longleftrightarrow}

\nc{\thla}{\twoheadleftarrow}

\nc{\hra}{\hookrightarrow}

\nc{\iso}{\overset{\sim}{\lra}}

\nc{\ssubset}{\underset{\not=}{\subset}}

\nc{\vac}{|0\rangle}

\nc{\Thm}[1]{Theorem~\ref{#1}}
\nc{\Prop}[1]{Proposition~\ref{#1}}
\nc{\Lem}[1]{Lemma~\ref{#1}}
\nc{\Cor}[1]{Corollary~\ref{#1}}
\nc{\Conj}[1]{Conjecture~\ref{#1}}
\nc{\Claim}[1]{Claim~\ref{#1}}
\nc{\Defn}[1]{Definition~\ref{#1}}
\nc{\Exa}[1]{Example~\ref{#1}}
\nc{\Rem}[1]{Remark~\ref{#1}}
\nc{\Note}[1]{Note~\ref{#1}}
\nc{\Quest}[1]{Question~\ref{#1}}
\nc{\Hyp}[1]{Hypoth\`ese~\ref{#1}}
\nen{thm}[1]{\label{#1}{\bf Theorem.\ } \em}{}
\nen{prop}[1]{\label{#1}{\bf Proposition.\ } \em}{}
\nen{lem}[1]{\label{#1}{\bf Lemma.\ } \em}{}
\nen{cor}[1]{\label{#1}{\bf Corollary.\ } \em}{}
\nen{conj}[1]{\label{#1}{\bf Conjecture.\ } \em}{}

\nen{claim}[1]{\label{#1}{\bf Claim.\ } \em}{}

\nen{defn}[1]{\label{#1}{\bf Definition.\ } }{}
\nen{exa}[1]{\label{#1}{\bf Example.\ } }{}
\nen{rem}[1]{\label{#1}{\em Remark.\ } }{}
\nen{note}[1]{\label{#1}{\em Note.\ } }{}
\nen{exer}[1]{\label{#1}{\em Exercise.\ } }{}
\nen{sket}[1]{\label{#1}{\em Sketch of proof.\ } }{}
\nen{quest}[1]{\label{#1}{\bf Question.\ } \em}{}

\nen{hyp}[1]{\label{#1}{\bf Hypoth\`ese.\ } \em}{}
\begin{document}
\setcounter{section}{-1}

\title[Weyl denominator identity]{Weyl denominator identity
for the affine Lie superalgebra $\fgl(2|2)\hat{}$}
\author[Maria Gorelik]{Maria Gorelik}

\address{Dept. of Mathematics, The Weizmann Institute of Science,
Rehovot 76100, Israel}
\email{maria.gorelik@weizmann.ac.il}
\thanks{Supported in part by ISF Grant No. 1142/07}

\begin{abstract}
We prove the Weyl denominator identity for the affine Lie superalgebra
$\fgl(2|2)\hat{}$ conjectured by V.~Kac and M.~Wakimoto in~\cite{KW}.
As it was pointed out in~\cite{KW}, 
this gives a new proof of the Jacobi identity for 
the number of presentations of a given integer 
as a sum of $8$ squares.
\end{abstract}

\maketitle

\section{Introduction}
The denominator identities for Lie superalgebras were formulated and
partially proven in
the paper of V.~Kac and M.~Wakimoto~\cite{KW}. 
In the same paper it was shown how various classical identities
in number theory as the number of representation of a given integer 
as a sum of $d$ squares can be obtained, for some $d$, by evaluation
of certain denominator identities.
The following cases are considered in the paper~\cite{KW}:

(a) basic Lie superalgebras, i.e. the finite-dimensional simple Lie
superalgebras, which have a reductive even part and admit an 
even non-degenerate invariant bilinear form;

(b) the affinization of basic Lie superalgebras with 
non-zero dual Coxeter number;

(c) the (twisted) affinization of a strange Lie superalgebras $Q(n)$;

(d) the affinization of $\fgl(2|2)$ (this is the smallest
 basic Lie superalgebras with zero dual Coxeter number).

Some of the cases (a), (b) are proven in~\cite{KW}; the proof is based on
combinatorics of root systems and a certain result
from representation theory. The rest of (a) was proven in~\cite{G1}
using only combinatorics of root systems. The rest of (b) was proven 
in~\cite{G} using (a) and the existence of Casimir operator.
The case (c) was proven in~\cite{Z} analytically. 
In the present  paper we prove (d), i.e. the identity
for the affine Lie superalgebra  ${\fgl}(2|2)\hat{}$
conjectured in~\cite{KW}, 7.1. The proof uses the existence
of Casimir operator and an idea of~\cite{Z}.

In order to write down the identity, we introduce the 
following infinite products after~\cite{DK}: 
for a parameter $q$ and a formal variable $x$ we set
$$(1+x)_q^{\infty}:=\prod_{n=0}^{\infty}(1+q^nx),\ \text{ and } \ 
(1-x)_q^{\infty}:=\prod_{n=0}^{\infty}(1-q^nx).$$
These infinite products converge
for any $x\in\mathbb{C}$ if the parameter $q$ is a real number $0<q<1$.
In particular, they are well-defined for $0<x=q<1$ and 
 $(1\pm q)_q^{\infty}:=\prod_{n=1}^{\infty}(1\pm q^n)$.

Take the formal variables $x,y_1,y_2$.
The denominator identity for ${\fgl}(2|2)\hat{}$ 
can be written in the following form
\begin{equation}\label{denom}\begin{array}{c}
\displaystyle\frac{(1-x)_q^{\infty}(1-qx^{-1})_q^{\infty}
(1-xy_1y_2)_q^{\infty}(1-q(xy_1y_2)^{-1})_q^{\infty}
\bigl((1-q)_q^{\infty}\bigr)^4}
{\displaystyle\prod_{i=1}^2(1+y_i)_q^{\infty}
(1+qy_i^{-1})_q^{\infty}(1+xy_i)_q^{\infty}
(1+qx^{-1}y_i^{-1})_q^{\infty}}=\\=
\displaystyle\frac{((1-q)_q^{\infty})^2}{(1-qy_1^{-1}y_2)_q^{\infty}
(1-qy_1y_2^{-1})_q^{\infty}}\cdot
\displaystyle\sum_{n=-\infty}^{\infty}(\frac{q^n}
{(1+q^ny_1)(1+q^ny_2)}-\frac{q^nx}
{(1+q^nxy_1)(1+q^nxy_2)}).
\end{array}
\end{equation}
Expanding the factor $\frac{((1-q)_q^{\infty})^2}{(1-qy_1^{-1}y_2)_q^{\infty}
(1-qy_1y_2^{-1})_q^{\infty}}$ 
in the region $q<|\frac{y_1}{y_2}|<q^{-1}$ we obtain (see~\Lem{factor})
$$\begin{array}{l}
\frac{((1-q)_q^{\infty})^2}{(1-qy_1^{-1}y_2)_q^{\infty}
(1-qy_1y_2^{-1})_q^{\infty}}=1+\sum_{n=1}^{\infty} f_n(\frac{y_1}{y_2}),\\
\text{ where }
f_n(y):=\bigl(y^n+y^{-n}-y^{n-1}-y^{1-n}\bigr)
\sum_{j=0}^{\infty}(-1)^jq^{(j+1)(j+2n)/2}\end{array}$$
and this gives the identity conjectured by V.~Kac and M.~Wakimoto.

The left-hand side of the identity represents the Weyl denominator $\hat{R}$
for the affine Lie superalgebra ${\fgl}(2|2)\hat{}$; the second
factor in the right-hand side is the analogue of  the right-hand side
of the denominator identity for affine Lie superalgebras
with non-zero dual Coxeter number. Note that  the denominator identity for
the affine Lie superalgebra ${\fsl}(2|2)\hat{}$
can be obtained from the denominator identity for ${\fgl}(2|2)\hat{}$
by taking $y_1=y_2$; as a result,
 the denominator identity for ${\fsl}(2|2)\hat{}$
is almost similar to the denominator identity for affine Lie superalgebras
with non-zero dual Coxeter number with one extra-factor
$(1-q)_q^{\infty}$ in the left-hand side (since the dimension of Cartan
subalgebra for $\fsl(2|2)$ is less by one than the dimension of Cartan
subalgebra for $\fgl(2|2)$).

As it is shown in~\cite{KW}, the evaluation of
this identity gives the following Jacobi identity~\cite{J}:
\begin{equation}\label{jacobi}
\square(q)^8=1+16\sum_{j,k=1}^{\infty} (-1)^{(j+1)k}k^3q^{jk},
\end{equation}
where $\square(q)=\sum_{j\in\mathbb{Z}}q^{j^2}$ and thus
the coefficient of $q^m$ in the power series expansion of $\square(q)^d$
is the number  of representation of a given integer 
as a sum of $d$ squares (taking into the account the order of summands).

In Section 1 we introduce notation. In Section 2  we prove
the identity~(\ref{denom}). In Section 3 we recall how to deduce
the Jacobi identity from the identity~(\ref{denom}).

\section{Notation}
\subsection{Root system}
Consider $V:=\mathbb{R}^5$ endowed by a bilinear form
and an orthogonal basis
$\vareps_1,\vareps_2,\delta_1,\delta_2,\delta$ such that
$(\vareps_i,\vareps_i)=1=-(\delta_i,\delta_i)$ for $i=1,2$ and 
$(\delta,\delta)=0$. Set
$$\beta_1:=\delta_1-\vareps_1,\ \alpha:=\vareps_1-\vareps_2,\ 
\beta_2:=\vareps_2-\delta_2,\ \gamma:=\beta_1+\alpha+\beta_2=
\delta_1-\delta_2.$$

The root system of $\fgl(2|2)$ is $\Delta_0=\{\pm\alpha,
\pm\gamma\}$, $\Delta_1=\{\pm\beta_i; \pm(\alpha+\beta_i)\}_{i=1,2}$. 
The affine root system is $\hat{\Delta}_i=\cup_{s\in\mathbb{Z}} (\Delta_i
+s\delta)$, $i=0,1$.

We consider the following sets of simple roots for  $\fgl(2|2)$
and  $\fgl(2|2)\hat{}$ respectively:
$$\Pi:=\{\beta_1,\alpha,\beta_2\},\ \text{ and }\  
\hat{\Pi}=\{\beta_1,\alpha,\beta_2,\delta-\gamma\}.$$
One has
$$\Delta_+=\{\alpha,\gamma;\beta_i,\alpha+\beta_i\}_{i=1,2},\ \ 
\hat{\Delta}_+=\Delta_+\cup\cup_{s=1}^{\infty} (\Delta
+s\delta),\ \ \ \hat{\rho}=\rho=-\frac{\beta_1+\beta_2}{2}.$$

Set
$$Q^+:=\sum_{\mu\in {\Pi}}\mathbb{Z}_{\geq 0}\mu,\ \ \ 
\hat{Q}^+=\sum_{\mu\in\hat{\Pi}}\mathbb{Z}_{\geq 0}\mu.$$

\subsubsection{}
For $\nu\in\Delta_0$  let
$s_{\nu}\in\Aut (V)$ be the reflection with respect to $\nu$, i.e.
$s_{\nu}(\lambda)=\lambda-\frac{(\lambda,\nu)}{(\nu,\nu)}\nu$.
The Weyl group $W$ of $\Delta_0$
takes form $W=W_{\alpha}\times W_{\gamma}$, where 
$W_{\alpha}$ (resp., $W_{\gamma}$) is generated by the reflection
$s_{\alpha}$ (resp., $s_{\gamma}$). 

For $\nu\in V$ introduce $t_{\nu}\in\Aut (V)$ 
by the formula
$$t_{\mu}(\lambda)=\lambda-(\lambda,\mu)\delta.$$
Then $t_{\mu}t_{\nu}=t_{\mu+\nu}$. For $\nu\in\Delta_0$ we denote by
$T_{\nu}$ the infinite cyclic group generated by $t_{\nu}$ and by
$\hat{W}_{\nu}$ the group generated by $s_{\nu}$ and $t_{\nu}$.
The Weyl group of ${\fgl}(2|2)\hat{}$ is 
$\hat{W}=\hat{W}_{\alpha}\times\hat{W}_{\gamma}$. 
Notice that $\delta$ and $\beta_1-\beta_2$ lie in the kernel of 
the bilinear form so these vectors are $\hat{W}$-stable.

For a subgroup $G$ of the Weyl
group we introduce the following operator:
$$\cF_G:=\sum_{w\in G} \sgn w\cdot w.$$

\subsection{Algebra $\cR$}
We are going to use notation of~\cite{G}, 1.4, which we recall below.

\subsubsection{}\label{R6}
Consider the space $\hat{\fh}^*=
V\oplus\mathbb{R}\Lambda_0$ and extend our bilinear form
by $(\Lambda_0,\delta)=1, (\Lambda_0,\Lambda_0)=(\Lambda_0,\vareps_i)
=(\Lambda_0,\delta_i)=0$ for $i=1,2$. 
The Weyl group $\hat{W}$ acts on $\hat{\fh}^*$ as follows: the reflections
act by the same formulas and the action of $t_{\mu}$
extends by the standard formula
$$t_{\mu}(\lambda)=\lambda+(\lambda,\delta)\mu-((\lambda,\mu)+
\frac{(\mu,\mu)}{2}(\lambda,\delta))\delta,\ \ \mu\in V,
\lambda\in\hat{\fh}^*$$

Call a {\em $\hat{Q}^+$-cone} a set of the form $(\lambda-\hat{Q}^+)$, 
where $\lambda\in\hat{\fh}^*$.

\subsubsection{}
For a formal sum of the form $Y:=\sum_{\nu\in\hat{\fh}^*} b_{\nu} e^{\nu},\ 
b_{\nu}\in\mathbb{Q}$ define the {\em support} of $Y$ by
$\supp(Y):=\{\nu\in\hat{\fh}^*|\ b_{\nu}\not=0\}$.
Let $\cR$ be a vector space over $\mathbb{Q}$,
spanned by the sums of the form
$\sum_{\nu\in \hat{Q}^+} b_{\nu} e^{\lambda-\nu}$, where $\lambda\in\hat{\fh}^*,\ 
b_{\nu}\in\mathbb{Q}$. In other words, $\cR$ consists of
the formal sums $Y=\sum_{\nu\in\hat{\fh}^*} b_{\nu}e^{\nu}$ with the support 
lying in a finite union of $\hat{Q}^+$-cones.

Clearly, $\cR$ has a structure of commutative algebra over 
$\mathbb{Q}$. If $Y\in \cR$ is such that $YY'=1$ for some $Y'\in\cR$,
we write $Y^{-1}:=Y'$.

\subsubsection{Action of the Weyl group}
For $w\in \hat{W}$ set $w(\sum_{\nu\in\hat{\fh}^*} b_{\nu}e^{\nu}):=
\sum_{\nu\in\hat{\fh}^*} b_{\nu}e^{w\nu}$. One has $wY\in\cR$ iff
$w(\supp Y)$ is a subset of a finite union of $\hat{Q}^+$-cones.

Let $W'$ be a subgroup of $\hat{W}$. 
Let $\cR_{W'}:=\{Y\in\cR|\ wY\in \cR \text{ for
each }w\in W'\}$. Clearly, $\cR_{W'}$ is a subalgebra of $\cR$.
 
\subsubsection{Infinite products}\label{infprod}
An infinite product of the form $Y=\prod_{\nu\in X}
(1+a_{\nu}e^{-\nu})^{r(\nu)}$, where $a_{\nu}\in \mathbb{Q},\
\ r(\nu)\in\mathbb{Z}_{\geq 0}$ and $X\subset \hat{\Delta}$ is such that 
the set $X\setminus\hat{\Delta}_+$ is finite, can be naturally viewed 
as an element of $\cR$; clearly, this element does not depend
on the order of factors. Let $\cY$ be the set of such infinite products.
For any $w\in \hat{W}$ the infinite product
$$wY:=\prod_{\nu\in X}(1+a_{\nu}e^{-w\nu})^{r(\nu)},$$
is again an infinite product of the above form, since, 
as one easily sees (\cite{G}, Lem. 1.2.8),  
the set $w\hat{\Delta}_+\setminus \hat{\Delta}_+$  is finite.
Hence $\cY$ is a $\hat{W}$-invariant multiplicative subset of
$\cR_{\hat{W}}$.
 
The elements of $\cY$ are invertible in $\cR$: using 
the geometric series we can expand  $Y^{-1}$ 
(for example, $(1-e^{\alpha})^{-1}=
-e^{-\alpha}(1-e^{-\alpha})^{-1}=-\sum_{i=1}^{\infty} e^{-i\alpha}$).

\subsubsection{The subalgebra $\cR'$}\label{cR'}
Denote by $\cR'$ the localization of $\cR_{\hat{W}}$ by $\cY$. By above,
$\cR'$ is a subalgebra of $\cR$. Observe that $\cR'\not\subset \cR_{\hat{W}}$:
for example, $(1-e^{-\alpha})^{-1}\in\cR'$, but 
$(1-e^{-\alpha})^{-1}=\sum_{j=0}^{\infty} e^{-j\alpha}\not\in \cR_{\hat{W}}$.
We extend the action of $\hat{W}$
from $\cR_{\hat{W}}$ to $\cR'$ by setting $w(Y^{-1}Y'):=(wY)^{-1}(wY')$
for $y\in\cY,\ Y'\in\cR_{\hat{W}}$.

An infinite product of the form $Y=\prod_{\nu\in X}
(1+a_{\nu}e^{-\nu})^{r(\nu)}$, where $a_{\nu}, X$ are as above
and $r(\nu)\in\mathbb{Z}$ lies in $\cR'$, and
$wY=\prod_{\nu\in X} (1+a_{\nu}e^{-w\nu})^{r(\nu)}$.
One has
$$\supp(Y)\subset\lambda'-\hat{Q}^+,\ \text{ where }
\lambda':=-\sum_{\nu\in X\setminus\hat{\Delta}_+: a_{\nu}\not=0} 
r_{\nu}\nu.$$

\begin{rem}{remsupp}
Set $q:=e^{-\delta},x:=e^{-\alpha},y_i:=e^{-\beta_i}$ and write
elements of $\cR'$ as power series in these variables. Since
$\{e^{-\nu},\nu\in\hat{\Pi}\}=\{x,y_1,y_2,q(xy_1y_2)^{-1}\}$,
the support of $Y\in\cR'$ correspond to the expansion
of $Y$ in the region $|q|<|xy_1y_2|; |x|,|y_1|,|y_2|<1$.
\end{rem}

\subsubsection{}\label{compex}
Let $W'$ be a subgroup of $\hat{W}$. 
For $Y\in\cR'$ we say that {\em $Y$ is $W'$-invariant
(resp., $W'$-anti-invariant)} if $wY=Y$
(resp., $wY=\sgn(w)Y$) for each $w\in W'$.

Let $Y=\sum a_{\mu} e^{\mu}\in\cR_{W'}$ be $W'$-anti-invariant. 
Then $a_{w\mu}=(-1)^{\sgn(w)}a_{\mu}$
for each $\mu$ and $w\in W'$. In particular,  
$W'\supp(Y)=\supp(Y)$, and, moreover, for each  
$\mu\in\supp(Y)$ one has $\Stab_{W'}\mu\subset\{w\in W'|\ \sgn(w)=1\}$.
The condition $Y\in \cR_{W'}$ is essential: for example, for
$W'=\{\id,s_{\alpha}\}$, the expressions $Y:=e^{\alpha}-e^{-\alpha}$,
$Y^{-1}=e^{-\alpha}(1-e^{-2\alpha})^{-1}$ are  $W'$-anti-invariant,
but $\supp(Y^{-1})=-\alpha,-3\alpha,\ldots$ is not 
$s_{\alpha}$-invariant.

Take $Y=\sum a_{\mu} e^{\mu}\in\cR_{W'}$.
The sum $\cF_{W'}(Y)=\sum_{w\in W'}\!\sgn(w) wY$ is an element of $\cR$
if for each $\mu$ the sum $\sum_{w\in W'}\sgn(w) a_{w\mu}$ is finite (i.e.,
$W'\mu\cap\supp(Y)$ is finite). In this case
$\cF_{W'}(Y)\in\cR$ and, writing $\cF_{W'}(Y)=\sum b_{\mu}e^{\mu}$, 
we obtain $b_{\mu}=\sum_{w\in W'}\sgn(w) a_{w\mu}$
so $b_{\mu}=\sgn(w)b_{w\mu}$
for each $w\in W'$. We conclude that 
$$Y\in\cR_{W'}\ \&\ 
\cF_{W'}(Y)\in\cR\ \Longrightarrow\ \left\{\begin{array}{l}
\cF_{W'}(Y)\in \cR_{W'};\\  \cF_{W'}(Y)
\text{ is $W'$-anti-invariant};\\ 
\supp (\cF_{W'}(Y))\text{ is  $W'$-stable}.
\end{array}\right.$$

Let us call a vector $\lambda\in\hat{\fh}^*$ {\em $W'$-regular} if $\Stab_{W'} 
\lambda=\{\id\}$.
Say that the orbit $W'\lambda$ is  $W'$-regular if $\lambda$ is $W'$-regular
(so the orbit consists of  $W'$-regular points). 
If $W'$ is an affine Weyl group, then for any $\lambda\in\hat{\fh}^*$
the stabilizer $\Stab_{W'} \lambda$ is either trivial or 
contains a reflection.
Thus for $W'=\hat{W}_{\alpha},\ \hat{W}_{\gamma}$ one has 
$$Y\in\cR_{W'}\ \&\ 
\cF_{W'}(Y)\in\cR\ \Longrightarrow\ \supp (\cF_{W'}(Y))\ \text{ is a union
of $W'$-regular orbits}.$$

\subsubsection{}\begin{rem}{remR}
For $Y\in\cR'$ the sum $\cF_{W'}(Y)$ is not always 
$W'$-anti-invariant: for example, for $W'=\{\id,s_{\alpha}\}$
 one has $\cF_{W'}((1-e^{-\alpha})^{-1})=
(1-e^{-\alpha})^{-1}-(1-e^{\alpha})^{-1}=1+2e^{-\alpha}+2e^{-2\alpha}+\ldots$
which is not $W'$-anti-invariant.
\end{rem}

\subsection{Another form of denominator identity}\label{anform}
Introduce  the following elements of $\cR$:
$$\begin{array}{lll}
R_0:=\prod_{\nu\in {\Delta}_{0,+}} 
(1-e^{-\nu}), &
R_1:=\prod_{\nu\in {\Delta}_{1,+}}
(1+e^{-\nu}), & R:=\frac{R_0}{R_1},\\
\hat{R}_0:=\prod_{\nu\in\hat{\Delta}_{0,+}} 
(1-e^{-\nu}), &
\hat{R}_1:=\prod_{\nu\in\hat{\Delta}_{1,+}}
(1+e^{-\nu}), & \hat{R}:=\frac{\hat{R}_0}{\hat{R}_1}.
\end{array}
$$
The products $Re^{\rho}$ and $\hat{R}e^{\rho}$ are $\hat{W}$-anti-invariant 
elements of $\cR'$ (see, for instance,~\cite{G}, Lem. 1.5.1).

\subsubsection{}\begin{lem}{factor}
In the region $q<|y|<q^{-1}$ one has
$$\frac{((1-q)_q^{\infty})^2}{(1-qy)_q^{\infty}
(1-qy^{-1})_q^{\infty}}=\sum_{n=1}^{\infty}
\bigl(y^n+y^{-n}-y^{n-1}-y^{1-n}\bigr)
\sum_{j=0}^{\infty}(-1)^jq^{(j+1)(j+2n)/2}$$
and this expression lies in $\cR$ for $y=e^{\beta_2-\beta_1}$.
\end{lem}
\begin{proof}
Consider the root system $\fsl(2|1)$ 
with the odd simple roots $\beta'_1,\beta_2'$ and
the even positive root $\alpha'=\beta_1'+\beta_2'$. Note that 
the corresponding element $\rho'$ is equal to zero.
Consider the corresponding affine root system, let $\delta'$ be
the minimal imaginary root and $\hat{W'}$ be its 
Weyl group. The affine denominator identity for $\fsl(2|1)$ 
written for $z:=e^{-\delta}$ takes form
$$\begin{array}{l}
\displaystyle\frac{(1-e^{-\alpha'})_z^{\infty}(1-ze^{\alpha'})_z^{\infty}
\bigl((1-z)_z^{\infty}\bigr)^2}
{\displaystyle\prod_{i=1}^2(1+e^{-\beta_i})_z^{\infty}
(1+ze^{\beta_i})_z^{\infty}}
=\sum_{n=-\infty}^{\infty}z^{n^2}\bigl(\frac{e^{n\alpha}}{1+z^ne^{-\beta_1}}
-\frac{e^{-n\alpha}}{1+z^ne^{\beta_2}}\bigr).
\end{array}$$
Both sides are well-defined for real $0<z<1$ and $\beta'_i$ such that
$e^{\beta'_i}\not=z^n$ for $n\in\mathbb{Z}$.  Taking $e^{\alpha'}=-1$
and $e^{-\beta'_1}:=-\xi$ we obtain
$e^{-\beta'_2}=e^{-\alpha}e^{\beta'_1}=\xi^{-1}$ and the evaluation gives
$$
\displaystyle\frac{2\bigl((1+z)_z^{\infty}(1-z)_z^{\infty}\bigr)^2}
{\displaystyle (1-\xi)_z^{\infty}(1+\xi^{-1})_z^{\infty}
(1-z\xi^{-1})_z^{\infty}(1+z\xi)_z^{\infty}}
=\sum_{n=-\infty}^{\infty}(-1)^nz^{n^2}\bigl(\frac{1}{1-z^n\xi}
-\frac{1}{1+z^n\xi}\bigr).$$
For $z^2=q, \xi^2=y$ we get
$$\displaystyle\frac{\bigl((1-q)_q^{\infty}\bigr)^2}
{\displaystyle (1-qy)_q^{\infty}
(1-qy^{-1})_q^{\infty}(1-y)}=2
\sum_{m=-\infty}^{\infty}(-1)^mq^{\frac{m^2+m}{2}}\frac{1}{1-q^my}.$$

For $m>0$ one has
$\frac{1}{1-q^my}=\sum_{k=0}^{\infty}q^{mk}y^k$ and
$\frac{1}{1-q^{-m}y}=-\sum_{k=1}^{\infty}q^{mk}y^{-k}$
so
$$\begin{array}{ll}
\displaystyle\frac{\bigl((1-q)_q^{\infty}\bigr)^2}
{\displaystyle (1-qy)_q^{\infty}
(1-qy^{-1})_q^{\infty}}&=1+(1-y)
\sum_{m=1}^{\infty}(-1)^m\sum_{k=0}^{\infty}(q^{\frac{m^2+m+2mk}{2}}y^k-
q^{\frac{m^2-m+2mk}{2}}y^{-k})\\
&=
1+\sum_{m=0}^{\infty}(-1)^m\sum_{k=1}^{\infty}(q^{\frac{(m+1)(m+2k)}{2}}(y^k
-y^{k-1}+y^{-k}-y^{1-k})
\end{array}$$
as required. One readily sees that the right-hand side of the above expression
lies in $\cR$ for $y=e^{\beta_2-\beta_1}$.
\end{proof}

\subsubsection{}
Set $x:=e^{-\alpha}, y_i:=e^{-\beta_i}$ 
for $i=1,2$ and $q:=e^{-\delta}$. Under this substitution,
the left-hand side of~(\ref{denom}) becomes $\hat{R}$ 
and, using~\Lem{factor}, we rewrite~(\ref{denom}) in the following form
\begin{equation}\label{denom1}
\hat{R}e^{\rho}=(1+\sum f_n)e^{-\rho}\cF_{\hat{W}_{\alpha}} 
\bigl(\frac{e^{\rho}}{(1+e^{-\beta_1})(1+e^{-\beta_2})}\bigr).
\end{equation}
Denote by $LHS$ (resp., $RHS$) the left-hand (resp., right-hand) side
of the identity~(\ref{denom1}).

The denominator identity for $\fgl(2|2)$ takes the form
\begin{equation}\label{Rrho}
\cF_{{W}_{\alpha}} 
\bigl(\frac{e^{\rho}}{(1+e^{-\beta_1})(1+e^{-\beta_2})}\bigr)=
Re^{\rho}=\cF_{W_{\gamma}}
\bigl(\frac{e^{\rho}}{(1+e^{-\beta_1})(1+e^{-\beta_2})}\bigr).
\end{equation}
so~(\ref{denom1}) can be rewritten as
$$\hat{R}e^{\rho}=(1+\sum f_n)\cF_{T_{\alpha}}(Re^{\rho}).$$

 In the sequel we need the following lemma.

\subsubsection{}
\begin{lem}{lemalphagamma}
If $\cF_{T_{\alpha}}(Re^{\rho})$ is well-defined (as an element of $\cR$), then
$$\cF_{T_{\alpha}}(Re^{\rho})=\cF_{T_{\gamma}}(Re^{\rho}).$$
\end{lem}
\begin{proof}
Note that $(\gamma-\alpha,\rho)=(\gamma-\alpha,\beta_i)=0$ 
for $i=1,2$ so $\frac{e^{\rho}}{(1-e^{-\beta_1})(1-e^{-\beta_2})}$
is invariant with respect to the action of $t_{\gamma-\alpha}$.
Therefore 
$$\cF_{{T}_{\alpha}} 
\bigl(\frac{e^{\rho}}{(1+e^{-\beta_1})(1+e^{-\beta_2})}\bigr)=\cF_{T_{\gamma}}
\bigl(\frac{e^{\rho}}{(1+e^{-\beta_1})(1+e^{-\beta_2})}\bigr).$$
Using the formula~(\ref{Rrho}), we obtain
$$\begin{array}{l}
\cF_{T_{\alpha}}(Re^{\rho})=
\cF_{{T}_{\alpha}} \circ \cF_{{W}_{\gamma}} 
\bigl(\frac{e^{\rho}}{(1+e^{-\beta_1})(1+e^{-\beta_2})}\bigr)=
=\cF_{{W}_{\gamma}}  \circ\cF_{{T}_{\alpha}}
\bigl(\frac{e^{\rho}}{(1+e^{-\beta_1})(1+e^{-\beta_2})}\bigr)\\
=\cF_{{W}_{\gamma}}  \circ\cF_{{T}_{\gamma}}
\bigl(\frac{e^{\rho}}{(1+e^{-\beta_1})(1+e^{-\beta_2})}\bigr)=
= \cF_{{T}_{\gamma}}\circ\cF_{{W}_{\gamma}} 
\bigl(\frac{e^{\rho}}{(1+e^{-\beta_1})(1+e^{-\beta_2})}\bigr)=
\cF_{T_{\gamma}}(Re^{\rho}),
\end{array}$$
as required.
\end{proof}

As a corollary,~(\ref{denom1}) can be rewritten as
$\hat{R}e^{\rho}=(1+\sum f_n)\cF_{T_{\gamma}}(Re^{\rho})$.

\section{Proof of the denominator identity}
\subsection{} 
\label{sectsupport}
By~\ref{infprod}, $LHS$ is an invertible element of $\cR'$. 
In this subsection we
show that $RHS$ is a well-defined element of $\cR$.

For $w\in\hat{W}_{\alpha}$ set
$$S_w:=\supp \bigl(w\bigl(\frac{e^{\rho}}{(1+e^{-\beta_1})
(1+e^{-\beta_2})}\bigr)\bigr).$$
One has
$$t_{n\alpha}\bigl(\frac{e^{\rho}}{(1+e^{-\beta_1})(1+e^{-\beta_2})}\bigr)=
\frac{e^{\rho}q^n}{(1+q^ne^{-\beta_1})(1+q^ne^{-\beta_2})}.$$

Then for $n>0$ one has
\begin{equation}
\label{Sw}
\begin{array}{l}
S_{\id}=\{\rho-k_1\beta_1-
k_2\beta_2\},\ \ \ S_{s_{\alpha}}= \{\rho-(1+k_1+k_2)\alpha-
k_1\beta_1-k_2\beta_2\},\\
S_{t_{n\alpha}}= \{\rho-n(1+k_1+k_2)\delta-k_1\beta_1-
k_2\beta_2\},\\ S_{t_{-n\alpha}}=\{\rho-n(1+k_1+k_2)\delta+(k_1+1)\beta_1+
(k_2+1)\beta_2\},\\
S_{t_{n\alpha}s_{\alpha}}=\{\rho-n(1+k_1+k_2)\delta-
(1+k_1+k_2)\alpha-k_1\beta_1-k_2\beta_2\},\\
S_{t_{-n\alpha}s_{\alpha}}=\{\rho-n(1+k_1+k_2)\delta+
(1+k_1+k_2)\alpha+k_1\beta_1+k_2\beta_2\},
\end{array}
\end{equation}
where $k_1,k_2\geq 0$. Observe that the above sets are pairwise
disjoint so the sum $\cF_{\hat{W}_{\alpha}}
\bigl(\frac{e^{\rho}}{(1+e^{-\beta_1})(1+e^{-\beta_2})}\bigr)$
is well-defined and its support lies in $\rho-\hat{Q}^+$.
Clearly, the sum $1+\sum f_n$ is well-defined.

One readily sees that 

\begin{equation}
\label{suppy1y2}
n\delta+k\beta_1-k\beta_2\in\hat{Q}^+\ \ \Longleftrightarrow\ \ 
|k|\leq n
\end{equation}

Thus the support of $1+\sum f_n$ lies in 
$\{0\}\cup \{-m\delta+k\beta_1-k\beta_2|\ m>0\}\cap -\hat{Q}^+$
(in particular, $(1+\sum f_n)\in\cR$). Hence $RHS$
 is a well-defined element of $\cR$.

\subsection{}\label{Y1Y2}
\begin{lem}{ai}
The expansion of $\frac{RHS}{LHS}$ in the 
region $|q|<|xy_1y_2|, |x|, |y_1|, |y_2|<1$ is of the form 
$1+\sum_{n=1}^{\infty}
\sum_{j=-n}^{n} a_{n,j} q^n (\frac{y_1}{y_2})^j$, where
 $a_{n,j}\in\mathbb{Z}$.
\end{lem}
\begin{proof}
Recall that $LHS=\hat{R}e^{\rho}$ and that
$\hat{R}\in\cY$ (see~\ref{infprod} for notation).
By~\ref{sectsupport}, $RHS\in\cR$. Therefore the fraction
$$Y:=\frac{RHS}{LHS}=\hat{R}^{-1}e^{-\rho}\cdot RHS$$
lies in $\cR$.

Clearly, $-\rho\in\supp(\hat{R}^{-1}e^{-\rho})\subset
(-\rho-\hat{Q}^+)$. Since 
$\supp (RHS)\subset \rho-\hat{Q}^+$, we conclude that 
$\supp (Y)\subset-\hat{Q}^+$. By~(\ref{Sw}) the coefficient of $e^{\rho}$
in $RHS$ is $1$; clearly, the coefficient of $e^{-\rho}$
in $\hat{R}^{-1}e^{-\rho}$ is also $1$, so the coefficient of $e^0=1$
in $Y$ is $1$. In the light of~\Rem{remsupp}, the required 
assertion is equivalent to the inclusion
\begin{equation}\label{suppY}
\supp(Y)\subset\{-n\delta+j(\beta_1-\beta_2)| n\geq 0, |j|\leq n\}.
\end{equation}

Retain notation of~\ref{R6}.
The element $\hat{\rho}_{\alpha}:=2\Lambda_0+\frac{\alpha}{2}$
is  the standard element for the corresponding copy of $\fsl_2\hat{}\subset 
\fgl(2|2)\hat{}$.
Recall that $\hat{R}=\frac{\hat{R}_0}{\hat{R}_1}$ 
(see~\ref{anform} for notation) so
$\hat{R}_1e^{\hat{\rho}_{\alpha}-\rho}=\hat{R}_0e^{\hat{\rho}_{\alpha}}\cdot
(\hat{R}e^{\rho})^{-1}$.
By~\ref{infprod}, $\hat{R}_1e^{\hat{\rho}_{\alpha}-\rho}$
belongs to $\cR_{\hat{W}}$. It is a standard fact
that $\hat{R}_0e^{\hat{\rho}_{\alpha}}$ is $\hat{W}_{\alpha}$-anti-invariant.
Recall that $\hat{R}e^{\rho}$ is $\hat{W}$-anti-invariant. Thus
 $\hat{R}_1e^{\hat{\rho}_{\alpha}-\rho}$ is a $\hat{W}_{\alpha}$-invariant
element of $\cR_{\hat{W}}$. One has 
$$\hat{R}_0e^{\hat{\rho}_{\alpha}}Y=
\hat{R}_1e^{\hat{\rho}_{\alpha}-\rho}\cdot RHS=
(1+\sum f_n)\cdot
\hat{R}_1e^{\hat{\rho}_{\alpha}-\rho}\cdot\cF_{\hat{W}_{\alpha}}
\bigl(\frac{e^{\rho}}{(1+e^{-\beta_1})
(1+e^{-\beta_2})}\bigr).$$

The $\hat{W}_{\alpha}$-invariance of 
$\hat{R}_1e^{\hat{\rho}_{\alpha}-\rho}$ gives
$$\hat{R}_1e^{\hat{\rho}_{\alpha}-\rho}\cdot\cF_{\hat{W}_{\alpha}}
\bigl(\frac{e^{\rho}}{(1+e^{-\beta_1})
(1+e^{-\beta_2})}\bigr)=\cF_{\hat{W}_{\alpha}}
\bigl(\frac{\hat{R}_1e^{\hat{\rho}_{\alpha}}}{(1+e^{-\beta_1})
(1+e^{-\beta_2})}\bigr)$$
so
$$\hat{R}_0e^{\hat{\rho}_{\alpha}}Y=(1+\sum f_n)\cdot \cF_{\hat{W}_{\alpha}}
\bigl(\frac{\hat{R}_1e^{\hat{\rho}_{\alpha}}}{(1+e^{-\beta_1})
(1+e^{-\beta_2})}\bigr).$$
By~\ref{sectsupport}, $\cF_{\hat{W}_{\alpha}}
\bigl(\frac{e^{\rho}}{(1+e^{-\beta_1})
(1+e^{-\beta_2})}\bigr)$ lies in $\cR$ so
$\cF_{\hat{W}_{\alpha}}
\bigl(\frac{\hat{R}_1e^{\hat{\rho}_{\alpha}}}{(1+e^{-\beta_1})
(1+e^{-\beta_2})}\bigr)$ lies in $\cR$.
By~\ref{infprod}, the term
$$\frac{\hat{R}_1e^{\hat{\rho}_{\alpha}}}{(1+e^{-\beta_1})
(1+e^{-\beta_2})}=e^{\hat{\rho}_{\alpha}}
\prod_{\beta\in\hat{\Delta}_{1,+}\setminus\{\beta_1,\beta_2\}}
(1+e^{-\beta})$$
lies in $\cR_{\hat{W}}$. Therefore,
in the light of~\ref{compex}, $\cF_{\hat{W}_{\alpha}}
\bigl(\frac{\hat{R}_1e^{\hat{\rho}_{\alpha}}}{(1+e^{-\beta_1})
(1+e^{-\beta_2})}\bigr)$ is a $\hat{W}_{\alpha}$-anti-invariant
element of $\cR_{\hat{W}_{\alpha}}$.
Observe that $\frac{y_1}{y_2}=e^{\beta_2-\beta_1}$
is $\hat{W}$-invariant so $f_n$ is $\hat{W}$-invariant.
Thus $(1+\sum f_n)\cF_{\hat{W}_{\alpha}}
\bigl(\frac{\hat{R}_1e^{\hat{\rho}_{\alpha}}}{(1+e^{-\beta_1})
(1+e^{-\beta_2})}\bigr)$ is a $\hat{W}_{\alpha}$-anti-invariant
element of $\cR_{\hat{W}_{\alpha}}$.
As a result, $\hat{R}_0e^{\hat{\rho}_{\alpha}}Y$ is
a $\hat{W}_{\alpha}$-anti-invariant element of $\cR_{\hat{W}_{\alpha}}$.

Write $Y=Y_1+Y_2$, where $Y_1=\sum_{n=0}^{\infty}
\sum_{j=-\infty}^{\infty} a_{n,j} q^n (\frac{y_1}{y_2})^j$
and $Y_2$ does not have monomials of the form $q^n (\frac{y_1}{y_2})^j$,
i.e $\supp(Y)=\supp(Y_1)\coprod\supp(Y_2)$.
One has $Y_i\in\cR$ 
because $\supp(Y_i)\subset \supp (Y)\subset-\hat{Q}^+$ ($i=1,2$).

Since $\frac{y_1}{y_2}=e^{\beta_2-\beta_1}$ is $\hat{W}$-invariant, 
$Y_1$ is a $\hat{W}$-invariant element of $\cR_{\hat{W}}$. Since
$\hat{R}_0e^{\hat{\rho}_{\alpha}}$ is 
a $\hat{W}_{\alpha}$-anti-invariant element of $\cR_{\hat{W}_{\alpha}}$,
the product $\hat{R}_0e^{\hat{\rho}_{\alpha}}Y_1$ is
a $\hat{W}_{\alpha}$-anti-invariant element of $\cR_{\hat{W}_{\alpha}}$.
By above, $\hat{R}_0e^{\hat{\rho}_{\alpha}}Y$ is
a $\hat{W}_{\alpha}$-anti-invariant element of $\cR_{\hat{W}_{\alpha}}$.
Hence $\hat{R}_0e^{\hat{\rho}_{\alpha}}Y_2$ is
a $\hat{W}_{\alpha}$-anti-invariant element of $\cR_{\hat{W}_{\alpha}}$.

Assume that $Y_2\not=0$. Recall that $\supp (Y_2)\subset-\hat{Q}^+$.
Let $\mu$ be a maximal element in $\supp (Y_2)$
with respect to the standard partial order 
$\mu\leq\nu$ if $(\nu-\mu)\in \hat{Q}^+$.
Then $\hat{\rho}_{\alpha}+\mu$ is a maximal element in
the support of $\hat{R}_0e^{\hat{\rho}_{\alpha}}Y_2$.
By~\ref{compex}, this support is the union of $\hat{W}_{\alpha}$-regular
orbits, so  $\hat{\rho}_{\alpha}+\mu$ is a maximal element
in a regular $\hat{W}_{\alpha}$-orbit (regularity means
that each element has the trivial stabilizer in $\hat{W}_{\alpha}$).
Since $\mu\in-\hat{Q}^+$ one has
$\frac{2(\mu,\alpha)}{(\alpha,\alpha)}\in\mathbb{Z}$.
Therefore
$\frac{2(\hat{\rho}_{\alpha}+\mu,\alpha)}{(\alpha,\alpha)}=1+(\mu,\alpha), \ 
\frac{2(\hat{\rho}_{\alpha}+\mu,\delta-\alpha)}
{(\delta-\alpha,\delta-\alpha)}=1+(\mu,\delta-\alpha)\ $
are positive integers so $(\mu,\alpha),\ (\mu,\delta-\alpha)\geq 0$.
Since  $\mu\in-\hat{Q}^+$ one has $(\mu,\delta)=0$ and thus
$(\mu,\alpha)=0$. 

The element $\hat{\rho}_{\gamma}:=-2\Lambda_0+\frac{\gamma}{2}$
is  the standard element for the corresponding copy of $\fsl_2\hat{}$.
Using~\Lem{lemalphagamma} we obtain
$$\hat{R}_0e^{\hat{\rho}_{\gamma}}Y=(1+\sum f_n)\cF_{\hat{W}_{\gamma}}
\bigl(\frac{\hat{R}_1e^{\hat{\rho}_{\gamma}}}{(1+e^{-\beta_1})
(1+e^{-\beta_2})}\bigr).$$
Repeating  the above  reasoning for $\hat{W}_{\gamma}$ we obtain
$(\mu,\gamma)=0$. Hence $(\mu,\alpha)=(\mu,\gamma)=0$ and
$\mu\in-\hat{Q}^+$. This implies $\mu=-m\delta+k(\beta_1-\beta_2)$,
which contradicts to the construction of $Y_2$.
Hence $Y_2=0$ so $Y=Y_1$
that is $\supp(Y)\subset \{-n\delta+j(\beta_1-\beta_2)\}$.
Combining the condition $\supp(Y)\subset-\hat{Q}^+$
and~(\ref{suppy1y2}), we obtain the required inclusion~(\ref{suppY}).
\end{proof}

\subsection{Evaluation}
\label{modpr}
By~\Lem{ai}, $\frac{RHS}{LHS}$ is a function of one variable
$y:=\frac{y_1}{y_2}$. In order to establish the identity $LHS=RHS$,
it is enough to verify that $\frac{RHS}{LHS}(y)=1$ for a fixed
$x$ and some $y_2, y_1$ satisfying $y_1=yy_2$. 
We will check this for $x=-1, y_2=y, y_1=y^2$, 
(i.e. $e^{-\alpha}=-1, e^{-\beta_1}=y^2, e^{-\beta_2}=y$).

One has $\frac{RHS}{LHS}=\hat{R}^{-1} (RHS\cdot e^{-\rho})$.
We write $RHS\cdot e^{-\rho}=AB$, where
$$
A:=\frac{((1-q)_q^{\infty})^2}{(1-qy)_q^{\infty}
(1-qy^{-1})_q^{\infty}},\ \ \ 
B:=e^{-\rho}\cdot\cF_{\hat{W}_{\alpha}}
\bigl(\frac{e^{\rho}}{(1+e^{-\beta_1})
(1+e^{-\beta_2})}\bigr),\ \ y=\frac{y_1}{y_2},
$$

\subsubsection{}
Recall that an infinite product $\prod_{i=1}^{\infty} (1+g_i(z))$, 
where $g_i(z)$ are holomorphic functions in $U\subset \mathbb{C}$ 
is called {\em normally convergent in $U$}
if $\sum g_i(z)$ normally  converges in $U$. By~\cite{R}, a
normally convergent infinite product converges to a 
function $g(z)$, which is holomorphic in $U$; moreover,
the set of zeros of $g(z)$ is the union of the sets of zeros of $1+f_i(z)$
and the order of each zero is the sum of the orders of the corresponding
zeros of $1+g_i(z)$.

The denominator of $A(y)$ normally  converges in any 
$U\subset X$, where $X\subset \mathbb{C}$ is a compact not containing
$0$. Thus $A(y)$ is a meromorphic function in the region $0<|y|$
with simple poles at the points $y=q^n$, $n\in\mathbb{Z}\setminus\{0\}$.

\subsubsection{}
The evaluation of $\hat{R}$ takes the form
$$\hat{R}(y)=\frac{2\prod_{n=1}^{\infty}
(1-q^n)^4(1+q^n)^2(1+q^{n-1}y^3)(1+q^{n}y^{-3})}
{\prod_{n=0}^{\infty}\prod_{s=1}^2(1+q^ny^{s})(1+q^{n+1}y^{-s})
(1-q^ny^{s})(1-q^{n+1}y^{-s})}.$$
All infinite product in the above expression normally  converge in any 
$U\subset X$, where $X\subset \mathbb{C}$ is a compact not containing
$0$. Therefore $\hat{R}(y)$ is a meromorphic function 
in the region $0<|y|$, and
\begin{equation}\label{LR}
\frac{A}{\hat{R}}(y)=\frac{(1-y)
\prod_{n=0}^{\infty}(1-q^{n}y^2)(1-q^{n+1}y^{-2})
\prod_{s=1}^2(1+q^ny^{s})(1+q^{n+1}y^{-s})}
{2\prod_{n=0}^{\infty}(1-q^{2n})^2(1+q^{n}y^3)
(1+q^{n+1}y^{-3})}
\end{equation}
is a meromorphic function in the region $0<|y|$ with simple poles,
the zero  of order two at $y=1$  and 
all other zeros of order one; the set of poles (resp., zeros)
is $P$ (resp., $Z$):
$$P:=\{y|\ y^3=-q^m\ \& \ y\not=-q^k\}_{k,m\in\mathbb{Z}},\ \ \ \
Z:=\{y|\ y^2=\pm q^m\}_{m\in\mathbb{Z}}.$$

One readily sees from~(\ref{LR}) that
\begin{equation}\label{ResA}
\lim_{y\to 1} (y-1)^{-2}\frac{A}{\hat{R}}(y)=2, \ \ \ \ \ \ 
\frac{A}{\hat{R}}(qy)=\frac{A}{\hat{R}}(y)\cdot\frac{q(1-qy)}{1-y}.
\end{equation}

\subsubsection{}
Recall that
$$B=\sum_{n=-\infty}^{\infty}\bigl(\frac{q^n}{(1+q^ne^{-\beta_1})
(1+q^ne^{-\beta_2})}-\frac{q^ne^{-\alpha}}{(1+q^ne^{-\beta_1-\alpha})
(1+q^ne^{-\beta_2-\alpha})}\bigr)$$
so the evaluation takes the form
\begin{equation}\label{B}\begin{array}{ll}
B(y)&=\sum_{n=-\infty}^{\infty}\bigl(\frac{q^n}{(1+q^ny)
(1+q^ny^2)}+\frac{q^n}{(1-q^ny)(1-q^ny^2)}\bigr)\\
&=\frac{1}{1-y}
\sum_{n=-\infty}^{\infty}\bigl(\frac{q^n}{1+q^ny}-\frac{q^ny}
{1+q^ny^2}+\frac{q^n}{1-q^ny}-\frac{q^ny}{1-q^ny^2}\bigr).
\end{array}
\end{equation}

Each point $y\in\mathbb{C}$  such that $y^2\not=\pm q^n$ 
for $n\in\mathbb{Z}$ has a neighborhood $U$ such that the above
sums converge absolutely and uniformly. 
Thus $B(y)$ is a meromorphic function in the region $0<|y|$ with  poles
at the points $\{y|\ y^2=\pm q^n\}_{n\in\mathbb{Z}}$, where 
all poles are simple except the pole of order two at $y=1$.
Let us verify that $B(y)=0$ for each $y\in P$. For $y^3=-q^k$,
$y\not\in\{-q^m\}$ one has
$$\frac{y}{1\pm q^ny^2}=\frac{y}{1\mp q^{n+k}y^{-1}}=\mp 
\frac{1}{1\mp q^{-n-k}y}$$
so $B(y)=0$. Hence $\frac{AB}{\hat{R}}(y)$
is a holomorphic function in the region $0<|y|$.

From the second formula of~(\ref{B})
one sees that $B(qy)=q^{-1}\frac{1-y}{1-qy}$; combining
with~(\ref{ResA}) we get $\frac{AB}{\hat{R}}(qy)=\frac{AB}{\hat{R}}(y)$.
Since $\frac{AB}{\hat{R}}(y)$ is a holomorphic function in the region $0<|y|$, 
this function is constant. One has
$$\lim_{y\to 1} (1-y)^2\cdot B(y)=
\lim_{y\to 1}(1-y)^2\frac{1}{(1-y)(1-y^2)}=\frac{1}{2}.$$
Using~(\ref{ResA}) we obtain $\frac{AB}{\hat{R}}(1)=1$
so $\frac{AB}{\hat{R}}(y)\equiv 1$ (for $0<|y|$). This completes
the proof of denominator identity.

\section{Application to Jacobi identity~(\ref{jacobi})}
Recall the Gauss' identity (which follows easily 
from the Jacobi triple product) 
$$\square(-q)=\frac{(1-q)_q^{\infty}}{(1+q)_q^{\infty}}.$$

The evaluation of the identity~(\ref{denom}) at $y_1=y_2=1$ gives
$$\begin{array}{l}
\displaystyle\frac{\bigl((1-x)_q^{\infty}(1-qx^{-1})_q^{\infty}\bigr)^2
\bigl((1-q)_q^{\infty}\bigr)^4}
{\displaystyle 4 ((1+q)_q^{\infty})^4
\bigl((1+x)_q^{\infty}
(1+qx^{-1})_q^{\infty}\bigr)^2}=
\displaystyle\sum_{n=-\infty}^{\infty} a_n,\\
\text{ where }
a_n:=\frac{q^n}
{(1+q^n)(1+q^n)}-\frac{q^nx}
{(1+q^nx)(1+q^nx)}.
\end{array}
$$
We divide both sides of the above identity
 by $\frac{(1-x)^2}{16}$ and take the limit $x\mapsto 1$; we get
$$\bigl(\frac{(1-q)_q^{\infty}}{(1+q)_q^{\infty}}\bigr)^8=1-16
\displaystyle\sum_{n=1}^{\infty}
\frac{q^n(q^{2n}-4q^n+1)}{(1+q^n)^4},$$
since
$$\lim_{x\to 1}\frac{a_0}{(x-1)^2}=\frac{1}{16},\ \ 
\lim_{x\to 1}\frac{a_n+a_{-n}}{(x-1)^2}=-
\frac{q^n(q^{2n}-4q^n+1)}{(1+q^n)^4}.$$

Using the expansion
$(a+1)^{-4}=\sum_{j=0}^{\infty} (-1)^j\frac{(j+1)(j+2)(j+3)}{6} a^j$, 
we obtain
$$\begin{array}{ll}
\square(-q)^8=\bigl(\frac{(1-q)_q^{\infty}}{(1+q)_q^{\infty}}\bigr)^8&=1+16
\displaystyle\sum_{n=1}^{\infty}\sum_{j=1}^{\infty}(-1)^j
j^3 q^{nj},\end{array}$$
which implies the required identity
$$\square(q)^8=1+16\displaystyle\sum_{n=1}^{\infty}
\sum_{j=1}^{\infty}(-1)^{j+nj}j^3 q^{nj}.$$


\end{document}